\documentclass{amsart}
\usepackage{amsmath,amssymb,amsthm,amsfonts,enumerate,array,color,lscape,fancyhdr,layout,pst-all}
\usepackage[english]{babel}
\usepackage[latin1]{inputenc}
\usepackage{young}

\usepackage[hcentermath]{youngtab}
\usepackage{graphicx}
\usepackage{float}
\usepackage{pstricks}

\newcounter{numberofremark}
\newcommand\nothing[1]{}
\usepackage[active]{srcltx}
\usepackage[hyperindex,bookmarksnumbered,plainpages]{hyperref}

\newcommand{\dcl}{\DeclareMathOperator}
\dcl\cdet{cdet} \dcl\Sp{Specm} \dcl\depth{depth} \dcl\im{Im} \dcl\Span{span} \dcl\Ker{Ker} \dcl\Specm{Specm}
\dcl\Supp{Supp} \dcl\codim{codim} \dcl\Y{Y} \dcl\gl{\mathfrak{gl}}    \dcl\U{U} \dcl\T{T}
\dcl\qdet{qdet} \dcl\sgn{sgn} \dcl\gr{gr} \dcl\diag{diag}
\dcl\g{\mathfrak{g}} \dcl\C{\mathbb C} \dcl\dd{{\mathrm d}}

\newcommand\sn{{\mathsf n}}
\newcommand\sr{{\mathsf r}}
\newcommand\sm{{\mathsf m}}

\newcommand\Ga{{\Gamma}}

\newlength\yStones
\newlength\xStones
\newlength\xxStones

\makeatletter
\def\Stones{\pst@object{Stones}}
\def\Stones@i#1{%
  \pst@killglue%
  \begingroup%
  \use@par%
  \setlength\xxStones{\xStones}%
  \expandafter\Stones@ii#1,,\@nil
  \endgroup
  \global\addtolength\xStones{0.6cm}%
  \global\addtolength\yStones{-7.5mm}}%
\def\Stones@ii#1,#2,#3\@nil{%
  \rput(\xxStones,\yStones){%
    \psframebox[framesep=0]{%
      \parbox[c][6mm][c]{11mm}{\makebox[11mm]{$#1$}}}}%
  \addtolength\xxStones{1.2cm}%
  \ifx\relax#2\relax\else\Stones@ii#2,#3\@nil\fi}
\makeatother

\def\Stone#1{\fbox{\makebox[10mm]{\strut#1}}\kern2pt}

\newtheorem{theorem}{Theorem}[section]
\newtheorem{lemma}[theorem]{Lemma}
\newtheorem{corollary}[theorem]{Corollary}
\newtheorem{proposition}[theorem]{Proposition}

\newtheorem{remark}[theorem]{Remark}
\newtheorem{notation}[theorem]{Notation}

\newtheorem{definition}[theorem]{Definition}

\begin{document}
\title{ Irreducible subquotients of generic Gelfand-Tsetlin modules over $U_{q}(\mathfrak{gl}_n)$}
\author{Vyacheslav Futorny}
\address{Instituto de Matem\'atica e Estat\'istica, Universidade de S\~ao
Paulo,  S\~ao Paulo SP, Brasil} \email{futorny@ime.usp.br}
\author{Luis Enrique Ramirez}
\address{Universidade Federal do ABC, Santo Andr\'e SP, Brasil} \email{luis.enrique@ufabc.edu.br}
\author{Jian Zhang}
\address{Instituto de Matem\'atica e Estat\'istica, Universidade de S\~ao
Paulo,  S\~ao Paulo SP, Brasil} \email{zhang@ime.usp.br}

\begin{abstract}
We provide a classification and explicit bases of tableaux of all irreducible subquotients of generic Gelfand-Tsetlin modules  over $U_{q}(\mathfrak{gl}_n)$ where $q\neq \pm 1$.
\end{abstract}


\maketitle
\section{Introduction}

Recently there has been a breakthrough in the theory of Gelfand-Tsetlin modules in the papers \cite{FGR1}, \cite{FGR2}, \cite{FGR3}, \cite{FGR}. In these papers new classes of simple 
$\mathfrak{gl}_n$-modules were constructed generalising a classical Gelfand-Tsetlin basis \cite{GT}, \cite{m:gtsb} for finite-dimensional representations. These new representations also have a basis consisting of  
Gelfand-Tsetlin tableaux but such tableaux are not necessarily eigenvectors of the Gelfand-Tsetlin subalgebra \cite{DFO3}.  This fact requires a modified action of the generators of the Lie algebra on this basis.
  Gelfand-Tsetlin representations are related to the theory of integrable systems  \cite{KW-1}, \cite{KW-2}, \cite{C1}, \cite{C2}, \cite{CE1}, \cite{CE2}, 
  general hypergeometric functions on the complex Lie group $GL(n)$, \cite{Gr1},\cite{Gr2}; solutions of the
Euler equation, \cite{FM},  \cite{Vi} among the others.

   The purpose of current paper is to study  the Gelfand-Tsetlin basis for quantum  $\mathfrak{gl}_n$ aiming to generalize the constructions above in the quantum case. Previously, partial results were obtained  for example in \cite{MT}, 
  \cite{UST} , \cite{UST2}, \cite{FHW}.   A general theory of Gelfand-Tsetlin modules for quantum  $\mathfrak{gl}_n$ was developed in \cite{FHR}. Even though quantization of the  Gelfand-Tsetlin basis for generic module in the non-root of unity case 
 may seem straightforward it  does require a very careful treatment which was done in this paper.  We also include a root of unity case.  
  
  Our main result is Theorem~\ref{the-main} which provides explicit construction of all irreducible generic Gelfand-Tsetlin modules with tableaux realization. In Section~\ref{section-root} we consider $q$ a root of unity  and apply our construction in this case. 
  Iy yields new explicit constructions of some finite dimensional irreducible modules.

\

\noindent{\bf Acknowledgements.}  
V.F. is
supported in part by the CNPq grant (301320/2013-6) and by the
Fapesp grant (2014/09310-5). 
J. Z. is supported by the Fapesp grant (2015/05927-0).

\section{Notation and conventions}

Throughout the paper we fix an integer $n\geq 2$. The ground field will be ${\mathbb C}$.  For $a \in {\mathbb Z}$, we write $\mathbb Z_{\geq a}$ for the set of all integers $m$ such that $m \geq a$. Similarly, we define $\mathbb Z_{< a}$, etc.   By $U_q$ we denote the quantum enveloping algebra of $\gl(n)$. We fix the standard Cartan subalgebra  $\mathfrak h$, the standard triangular decomposition and the corresponding basis of simple roots of $U_q$.  The weights of $U_q$ will be written as $n$-tuples $(\lambda_1,...,\lambda_n)$. For a commutative ring $R$, by ${\rm Specm}\, R$ we denote the set of maximal ideals of $R$. We will write the vectors in $\mathbb{C}^{\frac{n(n+1)}{2}}$ in the following form:
$$
L=(l_{ij})=(l_{n1},...,l_{nn}\ |\ l_{n-1,1},...,l_{n-1,n-1}\ |\ \cdots|l_{21}, l_{22}|l_{11}).
$$
For $1\leq j \leq i \leq n$, $\delta^{ij} \in {\mathbb Z}^{\frac{n(n+1)}{2}}$ is defined by  $(\delta^{ij})_{ij}=1$ and all other $(\delta^{ij})_{k\ell}$ are zero. For $i>0$ by $S_i$ we denote the $i$th symmetric group. Let $1(q)$ be the set of all complex $x$ such that $q^{x}=1$. Finally, for any complex number $x$, we set
\begin{align*}
(x)_q=\frac{q^x-1}{q-1},\quad
[x]_q=\frac{q^x-q^{-x}}{q-q^{-1}}.
\end{align*}

\section{Gelfand-Tsetlin modules}
Let $P$ be the free $\mathbb Z$-lattice of rank n with the
canonical basis $\{\epsilon_{1},\ldots,\epsilon_{n}\}$, i.e.
$P=\bigoplus_{i=1}^{n}\mathbb Z\epsilon_{i}$, endowed with  symmetric
bilinear form
$\langle\epsilon_{i},\epsilon_{j}\rangle=\delta_{ij}$.
Let $\Pi=\{\alpha_j=\epsilon_j-\epsilon_{j+1}\ |\ j=1,2,\ldots\}$ and
$\Phi=\{\epsilon_{i}-\epsilon_{j}\ |\ 1\leq i\neq j\leq n-1\}$. Then $\Phi$ realizes the root system
of type $A_{n-1}$ with $\Phi$ a basis of simple roots.

We define $U_{q}$ as a unital associative algebra generated by $e_{i},f_{i}(1\leq i
\leq n)$ and $q^{h}(h\in P)$ with the
following relations:
\begin{align}
q^{0}=1,\  q^{h}q^{h'}=q^{h+h'} \quad (h,h' \in  P),\\
q^{h}e_{i}q^{-h}=q^{\langle h,\alpha_i\rangle}e_{i}  ,\\
q^{h}f_{i}q^{-h}=q^{-\langle h,\alpha_i\rangle}f_{i} ,\\
e_{i}f_{j}-f_{j}e_{i}=\delta_{ij}\frac{q^{\alpha_i}-q^{-\alpha_i}}{q-q^{-1}} ,\\
e_{i}^2e_{j}-(q+q^{-1})e_ie_je_i+e_je_{i}^2=0  \quad (|i-j|=1),\\
f_{i}^2f_{j}-(q+q^{-1})f_if_jf_i+f_jf_{i}^2=0  \quad (|i-j|=1),\\
e_{i}e_j=e_je_i,\  f_if_j=f_jf_i  \quad (|i-j|>1).
\end{align}
The quantum special linear algebra $U_q(sl_n)$ is the subalgebra of $U_q$
generated by $e_i,\ f_i,\ q^{\pm \alpha_i}(i=1,2,\ldots, n-1)$.

%The algebra $U_q(gl_n)$ has a Hopf algebra structure with the
%coproduct $\Delta$, counit $\varepsilon$ and the antipode S defined by
%
%\begin{align}
%\Delta(q^h)=q^h\otimes q^h,\\
%\Delta(e_i)=e_i\otimes q^{-\alpha_i}+1\otimes e_i,\\
%\Delta(f_i)=f_i\otimes 1+q^{\alpha_i}\otimes f_k,\\
%\varepsilon(q^h)=1, \varepsilon(e_k)=\varepsilon(f_i)=0,\\
%S(q^h)=q^{-\lambda},\\
%S(e_i)=-e_iq^{\alpha_i},\\
%S(f_i)=-q^{-\alpha_i}f_i.
%\end{align}
%for $h\in P$, $i=1,2,\ldots, n-1$.
\begin{remark}[\cite{FRT}, Theorem 12]
$U_{q}$ has an alternative representation. It is isomorphic to the algebra generated by
$l_{ij}^{+}$, $l_{ji}^{-}$, $1\leq i \leq j \leq n$ subject to the relations
\begin{align}
RL_1^{\pm}L_2^{\pm}=L_2^{\pm}L_1^{\pm}R\\
RL_1^{+}L_2^{-}=L_2^{-}L_1^{+}R
\end{align}
where
$R=q\sum_{i}e_{ii}\otimes e_{ii}+\sum_{i\neq j}e_{ii}\otimes e_{jj}
+(q-q^{-1})\sum_{i<j}e_{ij}\otimes e_{ji}.$
The isomorphism between this two representations is given by
\begin{align*}
l_{ii}^{\pm}=q^{\pm\epsilon_i},\\
l_{i,i+1}^{+}=(q-q^{-1})q^{\epsilon_i}e_i,\\
l_{i+1,i}^{-}=(q-q^{-1})f_{i}q^{\epsilon_i}.
\end{align*}
\end{remark}

Let  for $m\leqslant n$, $\mathfrak{gl}_{m}$ be the Lie subalgebra
of $\gl (n)$ spanned by $\{ E_{ij}\,|\, i,j=1,\ldots,m \}$. We have the following chain
\begin{equation*}
\gl_1\subset \gl_2\subset \ldots \subset \gl_n.
\end{equation*}
It induces  the chain $U_1\subset$ $U_2\subset$ $\ldots$ $\subset
U_n$ for the universal enveloping algebras  $U_{m}=U(\gl_{m})$, $1\leq m\leq n$. If
By we denote by $(U_m)_q$ the quantum universal enveloping algebra of $\gl_m$. We have the following chain $(U_1)_q\subset$ $(U_2)_q\subset$ $\ldots$ $\subset
(U_n)_q$. Let $Z_{m}$ denotes the center of $(U_{m})_{q}$. The subalgebra of $U_q$ generated by $\{
Z_m\,|\,m=1,\ldots, n \}$ will be called the \emph{Gelfand-Tsetlin
subalgebra} of $U_q$ and will be denoted by  ${\Ga}_q$.
\begin{theorem}[\cite{FRT}, Theorem 14]\label{generators of the quantum center}
The center of $U_{q}(\mathfrak {gl}_m)$ is generated by the following $m+1$ elements

$$c_{mk}=\sum_{\sigma,\sigma'\in S_m}(-q)^{l(\sigma)+l(\sigma')}
l_{\sigma(1),\sigma'(1)}^{+}\cdots l_{\sigma(k),\sigma'(k)}^{+}l_{\sigma(k+1),\sigma'(k+1)}^{-}\cdots l_{\sigma(m),\sigma'(m)}^{-},$$
where $0\leq k \leq m$.
\end{theorem}
\begin{definition}
\label{definition-of-GZ-modules} A finitely generated $U$-module
$M$ is called a \emph{Gelfand-Tsetlin module (with respect to
$\Ga_q$)} if

\begin{equation}\label{equation-Gelfand-Tsetlin-module-def}
M=\bigoplus_{\sm\in\Sp\Ga_q}M(\sm),
\end{equation}

where $M(\sm)=\{v\in M| \sm^{k}v=0 \text{ for some }k\geq 0\}$. Equivalently,
\begin{equation}\label{equation-Gelfand-Tsetlin-module-def with characters}
M=\bigoplus_{\chi\in\Gamma_q^{*}}M(\chi)
\end{equation}
where
$M(\chi)=\{v\in M: \forall g\in\Gamma_q\text{, }\exists k\in\mathbb{Z}_{>0} \text{ such that } (g-\chi(g))^{k}v=0 \}$.\\
The {\it Gelfand-Tsetlin support} of $M$ is the set
$\Supp_{GT}(M):=\{\chi\in\Ga_q^{*}:M(\chi)\neq 0\}$.
\end{definition}

\begin{lemma}
Any submodule of a Gelfand-Tsetlin module over $U_q$ is a Gelfand-Tsetlin module.
\end{lemma}

\begin{proof}
Analogous to \cite{FGR2} Lemma 3.2.
\end{proof}

\section{Finite dimensional modules of $U_q$}

In this section we recall the quantum version of a classical result of  Gelfand and  Tsetlin which provides an explicit basis  for every irreducible finite dimensional $U_q$-module.

\begin{definition} For a vector $L=(l_{ij})$ in $\mathbb{C}^{\frac{n(n+1)}{2}}$, by $T(L)$ we will denote the following array with entries $\{l_{ij}:1\leq j\leq i\leq n\}$
\begin{center}

\Stone{\mbox{ \scriptsize {$l_{n1}$}}}\Stone{\mbox{ \scriptsize {$l_{n2}$}}}\hspace{1cm} $\cdots$ \hspace{1cm} \Stone{\mbox{ \scriptsize {$l_{n,n-1}$}}}\Stone{\mbox{ \scriptsize {$l_{nn}$}}}\\[0.2pt]
\Stone{\mbox{ \scriptsize {$l_{n-1,1}$}}}\hspace{1.5cm} $\cdots$ \hspace{1.5cm} \Stone{\mbox{ \tiny {$l_{n-1,n-1}$}}}\\[0.3cm]
\hspace{0.2cm}$\cdots$ \hspace{0.8cm} $\cdots$ \hspace{0.8cm} $\cdots$\\[0.3cm]
\Stone{\mbox{ \scriptsize {$l_{21}$}}}\Stone{\mbox{ \scriptsize {$l_{22}$}}}\\[0.2pt]
\Stone{\mbox{ \scriptsize {$l_{11}$}}}\\
\medskip
\end{center}
such an array will be called a \emph{Gelfand-Tsetlin tableau} of height $n$. A Gelfand-Tsetlin tableau of height $n$ is called \emph{standard} if
$l_{ki}-l_{k-1,i}\in\mathbb{Z}_{\geq 0}$ and $l_{k-1,i}-l_{k,i+1}\in\mathbb{Z}_{>0}$ for all $1\leq i\leq k\leq n-1$.
\end{definition}

The following theorem discribes the Gelfand-Tsetlin approach for simple finite dimensional $U_q$ modules with a given highest weight.

\begin{theorem}[\cite{UST2}, Theorem 2.11]\label{Theorem: quantum GT theorem}
Let $L(\lambda)$ be the finite dimensional irreducible module over $U_q$ of highest weight $\lambda=(\lambda_{1},\ldots,\lambda_{n})$. Then there exist a basis of $L(\lambda)$ consisting of all standard tableaux $T(L)$ with fixed top row $l_{nj}=\lambda_j-j$. Moreover,  the action of the generators of $U_q$ on $L(\lambda)$ is given by the  \emph{Gelfand-Tsetlin formulae}:
\begin{equation}\label{Gelfand-Tsetlin formulas}
\begin{split}
q^{\epsilon_{k}}(T(L))&=q^{a_k}T(L),\quad a_k=\sum_{i=1}^{k}l_{k,i}-\sum_{i=1}^{k-1}l_{k-1,i}+k,\ k=1,\ldots,n,\\
e_{k}(T(L))&=-\sum_{j=1}^{k}
\frac{\prod_{i} [l_{k+ 1,i}-l_{k,j}]_q}{\prod_{i\neq j} [l_{k,i}-l_{k,j}]_q}
T(L+\delta^{kj}),\\
f_{k}(T(L))&=\sum_{j=1}^{k}\frac{\prod_{i} [l_{k-1,i}-l_{k,j}]_q}{\prod_{i\neq j} [l_{k,i}-l_{k,j}]_q}T(L-\delta^{kj}).\\
\end{split}
\end{equation}
\end{theorem}

The next proposition gives the explicit action of the generators of $\Gamma_q$.
\begin{proposition}\label{eigenvalue}
The generator $c_{nk}$ of $\Gamma_q$ acts on $T(L)$ as a scalar multiplication by
$$\gamma_{nk}(L)=(k)_{q^{-2}}!(n-k)_{q^{-2}}!q^{k(k+1)+\frac{n(n-3)}{2}}\sum_{\tau}
q^{\sum_{i=1}^{k}l_{n\tau(i)}-\sum_{i=k+1}^{n}l_{n\tau(i)}}$$
where $\tau\in S_n$ is such that $\tau(1)<\cdots<\tau(k), \tau(k+1)<\cdots<\tau(n)$.
\end{proposition}
\begin{proof}
Analogous to \cite{JR} Theorem 5.1.
Choose a lowest weight vector $v=T(L)$ in $L(\lambda)$, the entries of $T(L)$ should satisfy $l_{ij}=l_{i+1,j+1}+1$ for any $i,j$. Note that the generators $l_{ij}^{+}$, $l_{ji}^{-}$ belong to the upper and lower Borel subalgebra generated by $e_i, f_i$  respectively.
The element $l_{\sigma(k+1),\sigma'(k+1)}^{-}\cdots l_{\sigma(n),\sigma'(n)}^{-}$ kills $v$ unless
$\sigma_{k+1}=\sigma_{k+1}',\ldots,\sigma_{n}=\sigma_{n}'$. But $\sigma_{1}\leq\sigma_{1}',\ldots,\sigma_{k}\leq\sigma_{k}'$, so one must have $\sigma_{i}=\sigma_{i}'$
for all $1\leq i \leq n$ in the action of $c_{nk}$ on $v$. We thus have
\begin{align}
c_{nk}v&=\sum_{\sigma\in S_n}q^{2l(\sigma)}q^{a_{\sigma(1)}+\cdots +a_{\sigma(k)}-a_{\sigma(k+1)}-\cdots-a_{\sigma(n)}}v
\end{align}
where
\begin{align*}
a_{\sigma(i)}&=\sum_{j=1}^{\sigma(i)}l_{\sigma(i),j}-\sum_{j=1}^{\sigma(i)-1}l_{\sigma(i)-1,j}+\sigma(i)\\
&=l_{\sigma(i),1}+1\\
&=\lambda_{n+1-\sigma(i)}.
\end{align*}

Then
\begin{align*}
\gamma_{nk}(\lambda)=&\sum_{\sigma\in S_n}q^{2l(\sigma)}q^{a_{\sigma(1)}+\cdots +a_{\sigma(k)}-a_{\sigma(k+1)}-\cdots-a_{\sigma(n)}}\\
&=\sum_{\sigma\in S_n}q^{2l(\sigma)}q^{\lambda_{n+1-\sigma(1)}+\cdots +\lambda_{n+1-\sigma(k)}-\lambda_{n+1-\sigma(k+1)}-\cdots-\lambda_{n+1-\sigma(n)}}\\
&=\sum_{\sigma\in S_n}q^{n(n-1)-2l(\sigma)}q^{\lambda_{\sigma(1)}+\cdots +\lambda_{\sigma(k)}-\lambda_{\sigma(k+1)}-\cdots-\lambda_{\sigma(n)}}.
\end{align*}

Let $\tau$ be a permutation in $S_n$ such that $\tau(1)<\cdots<\tau(k),\  \tau(k+1)<\cdots<\tau(n)$. One has that
\begin{align*}
\gamma_{nk}(\lambda)&=(k)_{q^{-2}}!(n-k)_{q^{-2}}!q^{n(n-1)}\sum_{\tau}q^{-2l(\tau)}q^{\lambda_{\tau(1)}+\cdots +\lambda_{\tau(k)}-\lambda_{\tau(k+1)}-\cdots-\lambda_{\tau(n)}}\\
&=(k)_{q^{-2}}!(n-k)_{q^{-2}}!q^{n(n-1)}\sum_{\tau}
q^{-2\sum_{i=1}^{k}(\tau(i)-i)+\sum_{i=1}^{k}(l_{n\tau(i)}+\tau(i))-\sum_{i=k+1}^{n}(l_{n\tau(i)}+\tau(i))}\\
&=(k)_{q^{-2}}!(n-k)_{q^{-2}}!q^{k(k+1)+\frac{n(n-3)}{2}}
\sum_{\tau}q^{\sum_{i=1}^{k}l_{n\tau(i)}-\sum_{i=k+1}^{n}l_{n\tau(i)}}
\end{align*}
\end{proof}
\begin{corollary}
The generator $c_{mk}$ of $\Gamma_q$ acts on $T(L)$ as multiplication by
\begin{equation}\label{eigenvalues of gamma_mk}
\gamma_{mk}(\lambda)=(k)_{q^{-2}}!(m-k)_{q^{-2}}!q^{k(k+1)+\frac{m(m-3)}{2}}\sum_{\tau}
q^{\sum_{i=1}^{k}l_{m\tau(i)}-\sum_{i=k+1}^{m}l_{m\tau(i)}}
\end{equation}
where $\tau\in S_m$ is such that $\tau(1)<\cdots<\tau(k), \tau(k+1)<\cdots<\tau(m)$.
\end{corollary}
\begin{proof}
Follows directly from Theorem \ref{eigenvalue} and the fact that eigenvalues of $\gamma_{nk}$ depend only on the $n$-th row of the tableau.
\end{proof}

\section{Generic Gelfand-Tsetlin modules of $U_q$}

Recall that $1(q)$ stands for the set of all complex $x$ such that $q^{x}=1$.
\begin{definition}\label{Definition: q-generic tableau}
A Gelfand-Tsetlin tableau $T(L)$ is called \emph{$q$-generic} if
it satisfies the following defining conditions:
$$l_{ij}-l_{ik}\notin \frac{1(q)}{2}+ \mathbb{Z} \text{ for all }1\leq i\leq n \text{ and } k\neq j.$$
By ${\mathcal B}(T(L))$ we will denote the set of all Gelfand-Tsetlin tableaux $T(R)$ of height $n$ satisfying $r_{nj}=l_{nj}$ and  $r_{ij}-l_{ij}\in\mathbb{Z}$ for $1\leq j\leq i \leq n-1$.
\end{definition}

\begin{theorem}[\cite{MT} Theorem 2]\label{Theorem: GT theorem q-generic}
Let $T(L)$ be a generic tableau, the vector space $V(T(L)) = \Span {\mathcal B}(T(L))$ has  a structure of a $U_q$-module of finite length with action of the generators of $U_q$ given by the Gelfand-Tsetlin formulae (\ref{Gelfand-Tsetlin formulas}).
\end{theorem}

\begin{proposition}\label{Proposition: separation of tableaux}
The Gelfand-Tsetlin subalgebra $\Gamma_q$ separate the tableaux in $V(T(L))$. That is, for any two different tableaux  in $V(T(L))$, there exists an element in $\Gamma_{q}$ with different eigenvalues corresponding to the tableaux.
\end{proposition}
\begin{proof}
Let $T(R)$ and $T(S)$ be two tableaux with different $m$-th row. Assume $T(R)$ and $T(S)$ have the same eigenvalue for any element in $\Gamma_q$ . It is easy to see from (\ref{eigenvalue}) that $(q^{2s_{m1}},\ldots,q^{2s_{mm}})$ is a permutation of $(q^{2r_{m1}},\ldots,q^{2r_{mm}})$. Therefore, for any $r_{mi}$, there exist $j$ such that $q^{2r_{mi}}=q^{2s_{mj}}$, which implies that $r_{mi}-s_{mj}\in\frac{1(q)}{2}$. This lead to $i=j$ and $r_{mi}=s_{mj}$ which is a contradiction.
\end{proof}

\subsection{Classification of irreducible generic Gelfand-Tsetlin $U_q$-modules}

We recall the following result of Mazorchuk and Turowska.

\begin{theorem}[\cite{MT} Proposition 2]\label{unique GT}
If $\sn\in\Sp\Gamma$ is generic, then there exists a unique irreducible Gelfand-Tsetlin module $N$ such that $N(\sn)\neq 0$.
\end{theorem}

\begin{definition} \label{def-cont}
If $T(R)$ is a $q$-generic tableau and $\sr\in\Sp\Gamma_q$ corresponds to $R$ then, the unique module $N$ such that $N(\sr) \neq 0$ is called the \emph{irreducible Gelfand-Tsetlin module containing $T(R)$}, or simply,  the \emph{irreducible module containing $T(R)$}.
\end{definition}

This section is devoted to  an explicit construction of the irreducible Gelfand-Tsetlin module containing $T(R)$ for every $q$-generic tableau $T(R)$.

For convenience we introduce and recall some notation.

\begin{notation} \label{notation}Let $T(L)$ be a fixed tableau of height $n$.
\begin{itemize}
\item[(i)] $\mathcal{B}(T(L)):=\{T(L+z): z\in \mathbb{Z}^{\frac{n(n-1)}{2}}\}$.
\item[(ii)] $V(T(L)) = \Span \mathcal{B}(T(L))$.
\item[(iii)] For any $T(R)\in \mathcal{B}(T(L))$ and for any $1<p\leq n$, $1\leq s\leq p$ and $1\leq u\leq p-1$ we define:
\begin{itemize}
\item[(a)] $\omega_{p,s,u}(T(R)):=r_{p,s}-r_{p-1,u}$.
\item[(b)] $\Omega(T(R)):=\{(p,s,u): \omega_{p,s,u}(T(R))\in \frac{1(q)}{2}+\mathbb{Z}\}$
\item[(c)] $\Omega^{+}(T(R)):=\{(p,s,u): \omega_{p,s,u}(T(R))\in \frac{1(q)}{2}+\mathbb{Z}_{\geq 0}\}$
\item[(d)] $\mathcal{N}(T(R)):=\{T(S)\in \mathcal{B}(T(L)): \Omega^{+}(T(R))\subseteq \Omega^{+}(T(S))\}$
\item[(e)] $ W(T(R)):=\Span \mathcal{N}(T(R))$
\item[(f)] $U_q\cdot T(R)$: the $U_q$-submodule of $V(T(L))$ generated by $T(R)$
\end{itemize}
\end{itemize}
\end{notation}

\subsection{Submodule generated by a single tableau}
In order to find an explicit basis of every irreducible generic module, we first find a  basis of  $U_q\cdot T(R)$ for any tableau $T(R)$ in $\mathcal{B}(T(L))$.

\begin{definition}
Given $T(Q)$ and $T(R)$ in $\mathcal{B}(T(L))$, we write $T(R)\preceq_{(1)}T(Q)$ if there exist $g\in\mathfrak{gl}(n)$ such that $T(Q)$ appears with nonzero coefficient in the decomposition of $g\cdot T(R)$ into a linear combination of tableaux. For any $p\geq 1$ we write  $T(R)\preceq_{(p)} T(Q)$ if there exist tableaux $T(L^{(1)})$,..., $T(L^{(p)})$, such that $$T(R)=T(L^{(0)})\preceq_{(1)}T(L^{(1)})\preceq_{(1)}\ldots\preceq_{(1)}T(L^{(p)})=T(Q).$$
\end{definition}

As an immediate consequence of the definition of $\preceq_{(p)}$ we have the following.
\begin{lemma}\label{properties of relation less or equal}
If $T(Q)$, $T(Q^{(0)})$, $T(Q^{(1)})$ and $T(Q^{(2)})$ are tableaux in $\mathcal{B}(T(L))$ then:

\begin{itemize}
\item[(i)] $T(Q^{(0)})\preceq_{(p)} T(Q^{(1)})$ and $T(Q^{(1)})\preceq_{(q)} T(Q^{(2)})$ imply \\ $T(Q^{(0)})\preceq_{(p+q)} T(Q^{(2)})$.
\item[(ii)] $T(Q)\preceq_{(1)} T(Q)$.
\end{itemize}
\end{lemma}

The next theorem discribes the submodule of $V(T(L))$ generated by a fixed tableau $T(R)$.

\begin{theorem}\label{basis of submodule generated by T}
Let $T(L)$ be $q$-generic tableau, $T(R)$ and $T(S)$ be in $\mathcal{B}(T(L))$.
\begin{itemize}
\item[(i)]
The Gelfand-Tsetlin formulas endow $ W(T(R))$ with a $U_q$-module structure.

\item[(ii)]
$U_q\cdot T(R)= W(T(R))$. In particular, $\mathcal{N}(T(R))$ forms a basis of   $U_q\cdot T(R)$, and the action of $U_q$ on  $U_q\cdot T(R)$ is given by the Gelfand-Tsetlin formulas.
\item[(iii)]$U_q\cdot T(R)=  U_q\cdot T(S)$ if and only if $\Omega^{+}(T(S))=\Omega^{+}(T(R))$.
\item[(iv)]$U_q \cdot T(R)=V(T(L))$ whenever  $\Omega^{+}(T(R))=\emptyset$.
\item[(v)]Every submodule of $V(T(L))$ is finitely generated.
\end{itemize}
\end{theorem}

\begin{proof}
(i) In order to prove that $W(T(R))$ is a submodule, it is enough to prove $U\cdot T(S)\subseteq  W(T(R))$ for any $T(S)\in\mathcal{N}(T(R))$. We will show $g\cdot T(S)$ is in $W(T(R))$ for every (standard) generator of  $U_{q}$.

Suppose $g=e_{k}$ for some $1\leq k\leq n-1$. By the Gelfand-Tsetlin formulas, we have
$$e_{k}(T(S))=-\sum_{j=1}^{k}\frac{\prod_{i} [s_{k+ 1,i}-s_{k,j}]_q}{\prod_{i\neq j} [s_{k,i}-s_{k,j}]_q}T(S+\delta^{kj})$$

If $e_{k}(T(S))\notin  W(T(R))$, then there exist $k$ and  $j$ such that $T(S)\in \mathcal{N}(T(R))$ but $T(S+\delta^{kj})$ $\notin\mathcal{N}(T(R))$. That implies
\begin{align*}\Omega^{+}(T(R))\subseteq \Omega^{+}(T(S)),
\text{ and }
\Omega^{+}(T(R))\nsubseteq\Omega^{+}(T(S+\delta^{kj})),
\end{align*}
\noindent
Hence, there exists $(p,s,u)\in\Omega^{+}(T(R))$ such that $\omega_{p,s,u}(T(S))\in\frac{1(q)}{2}+\mathbb{Z}_{\geq 0}$ and\\ $\omega_{p,s,u}(T(S+\delta^{kj}))$ $\notin \frac{1(q)}{2} +\mathbb{Z}_{\geq 0}$. The latter holds only  in two cases:
$$
(p,s,u)\in \{(k,j,u), (k+1,s,j): 1\leq u\leq k-1; 1\leq s\leq k+1\}.
$$
Note that if neither of these two cases hold, we have $\omega_{p,s,u}(T(R+\delta^{kj}))=\omega_{p,s,u}(T(S))$. We consider now each of the two cases separately.

\begin{enumerate}[(i)]
\item[(a)] Suppose $(p,s,u)=(k,j,u)$. Then $\omega_{k,j,u}(T(S))=s_{kj}-s_{k-1,u}\in\frac{1(q)}{2}+\mathbb{Z}_{\geq 0}$ and $\omega_{k,j,u}(T(S+\delta^{kj}))=(s_{kj}+1)-s_{k-1,u}\notin \frac{1(q)}{2}+\mathbb{Z}_{\geq 0}$, which is impossible.
\item[(b)] Suppose $(p,s,u)=(k+1,s,j)$. Then $\omega_{k+1,s,j}(T(S))=s_{k+1,s}-s_{ki}\in \frac{1(q)}{2}+\mathbb{Z}_{\geq 0}$ and  $\omega_{k+1,s,i}(T(S+\delta^{ki}))=s_{k+1,s}-(s_{ki}+1)\notin\frac{1(q)}{2}+\mathbb{Z}_{\geq 0}$. Hence $s_{k+1,s}-s_{k,i}=0$ and then the coefficient of $T(S+\delta^{ki})$ in the decomposition of $e_{k}(T(S))$ is
    $$-\frac{\prod_{i} [s_{k+ 1,i}-s_{k,j}]_q}{\prod_{i\neq j} [s_{k,i}-s_{k,j}]_q}=0.$$
\end{enumerate}
Therefore, the tableaux that appear with nonzero coefficients in the decomposition of $e_{k}(T(S))$
are elements of $N(T(R))$. Hence,  $e_{k}(T(S))\in  W(T(R))$.

The proof that  $f_{k}(T(S))\in  W(T(R))$ is analogous to the one of  $e_{k}(T(S))$ $\in W(T(R))$. The case
$ q^{\epsilon_{k}}$ is trivial because $ q^{\epsilon_{k}}$ acts as a  multiplication by a scalar on $T(S)$ and $T(S)\in\mathcal{N}(T(R))\subseteq  W(T(R))$.

(ii) The Gelfand-Tsetlin subalgebra separate tableaux in $\mathcal{B}(T(L))$, it is sufficient to prove that for any $T(S)\in W(T(R))$, $T(S)\preceq_{(p)} T(R)$  for some $p\in \mathbb{Z}_{> 0}$. Let $T(S)=T(R+z)$, we prove the statement by induction on $\sum_{1\leq j\leq i< n} |z_{ij}|$. when $\sum_{1\leq j\leq i< n} |z_{ij}|=1$, there exist $z_{ij}=\pm 1$, all other entries are zero. We consider each case separately.

\begin{enumerate}[(i)]
\item [(a)] Suppose $z_{ij}=1$. Then the coefficient of $T(S)$ in $e_i T(R)$ is

$$-\frac{\prod_{i} [r_{i+ 1,k}-r_{i,j}]_q}{\prod_{i\neq j} [r_{j,k}-r_{i,j}]_q}.$$
If there exist $[r_{i+1,k}-r_{i,j}]_q=0$, one has $s_{i+1,k}-s_{i,j}=\frac{1(q)}{2}-1$, then $T(S)\notin W(T(R))$.
Thus $r_{i+1,k}-r_{i,j}\neq 0$ for any $k$ which implies $T(S)\preceq_{(1)}T(R)$
\item [(b)] Suppose $z_{ij}=-1$. Similarly the coefficient of $T(S)$ in $f_i T(R)$ is not zero.
\end{enumerate}

When $\sum_{1\leq j\leq i< n} |z_{ij}|>1$, It is sufficient to proof the following statement. Let $z_{i,j_0},z_{i+1,j_1},\cdots,z_{i_k, j_k}$ be the nonzero elements such that $r_{i+t,j_t}-r_{i+t',j_{t'}}\in\frac{1(q)}{2}+\mathbb{Z}$ for any $1\leq t, t'\leq k$, then there exist $T(S')=T(R+z')$ such that $\Omega^{+}T(S)\subseteq\Omega^{+}T(S')\subseteq\Omega^{+}T(R)$ and $|z'_{ij}|\leq|z_{ij}|$ . Let $t$ be the maximal number such that all the numbers $z_{i,j_0},\ldots,z_{i+t,j_t}$  have the same sign, then $\Omega^{+}(T(R))\subseteq\Omega^{+}(T(S- \sum_{s=0}^{t}\delta^{i+s,j_s}))\subseteq\Omega^{+}(T(S))$ if the sign is positive. $\Omega^{+}(T(R))\subseteq\Omega^{+}(T(S+\sum_{s=0}^{t}\delta^{i+s,j_s}))\subseteq\Omega^{+}(T(S))$ if the sign is negative. By induction one has that $T(S)\preceq_{(p)} T(R)$.

(iii) (iv) and (v) are easy to see from (i) and (ii).
\end{proof}

\section{Main results}
\begin{definition}
For any $q$-generic tableau $T(L)$, the {\em block associated with $T(L)$} is the set of all Gelfand-Tsetlin $U_q$-modules with Gelfand-Tsetlin support contained in $Supp_{GT}(V(T(L)))$.\\
Also, for any $T(R)\in\mathcal{B}(T(L))$, $1<p\leq n$ and $1\leq u\leq p-1$, define $d_{pu}(T(R))$ to be the number of distinct elements in $
\{v_{p,s,u}(T(R))\ |\ (p,s,u)\in\Omega(T(R))\}.
$, where $\omega_{p,s,u}(T(R))=u_{p,s,u}(T(R))+v_{p,s,u}(T(R)),$ with $u_{p,s,u}(T(R))\in \frac{1(q)}{2}$ and $v_{p,s,u}(T(R))\in\mathbb{Z}$.
\end{definition}

Now we are ready to give the main theorem in the paper.

\begin{theorem}\label{the-main}
Let $T(L)$ be $q$-generic tableau, $T(R)\in \mathcal{B}(T(L))$.
\begin{itemize}
\item[(i)]
The irreducible module  containing $T(R)$ has a basis of tableaux
$$\mathcal{I}(T(R))=\{T(S)\in\mathcal{B}(T(R)): \Omega^{+}(T(S))=\Omega^{+}(T(R))\}.$$
The action of $U_q$ on this irreducible module is given by the Gelfand-Tsetlin formulas (\ref{Gelfand-Tsetlin formulas}).
\item[(ii)]
The number of irreducible modules in the block associated with $T(L)$ is:
$$\prod_{1\leq u\leq p-1< n}(d_{pu}(T(L))+1).$$
In particular, $V(T(L))$ is irreducible if and only if $d_{pu}(T(L))=0$ for any $p$ and $u$, or equivalently, if and only if $\Omega(T(L))=\emptyset$.
\end{itemize}
\end{theorem}

\begin{proof}
\begin{itemize}
\item[(i)] For each tableau $T(R)$, we have an explicit construction of the  module containing $T(R)$ (recall Definition \ref{def-cont}):
$$M(T(R)):=U\cdot T(R)/\left(\sum U\cdot T(S)\right)$$
where the sum is taken over tableaux $T(Q)$ such that $\Omega^{+}(T(R))\subsetneq\Omega^{+}(T(S))$ and $U\cdot T(S)$ is a proper submodule of $U\cdot T(R)$.The module $M(T(R))$ is simple. Indeed, this follows from the fact that for any nonzero tableau $T(S)$ in $M(T(R))$ we have $U\cdot T(S)=U\cdot T(R)$ and, hence,  $T(S)$ generates $M(T(R))$.
By Theorem \ref{basis of submodule generated by T}, $\mathcal{I}(T(R))$ is a basis for $M(T(R))$.
\item[(ii)]
The irreducible modules are in one-to-one correspondence with the subsets of $\Omega(T(L))$ of the form $\Omega^{+}(T(L+z))$. For any $T(R)\in\mathcal{B}(T(L))$, we can decompose $\Omega(T(R))$ into a disjoint union
$\Omega(T(R))=\bigsqcup_{p,u}\Omega_{pu}(T(R))$, where $$\Omega_{p,u}(T(R))=\{(p,1,u),(p,2,u),\ldots,(p,p,u)\}\cap \Omega(T(R)).$$ Now, if $\Omega^{+}_{p,u}(T(R)):=\Omega_{p,u}(T(R))\cap\Omega^{+}(T(R))$, one can  write $\Omega^{+}(T(R))=\bigsqcup_{p,u}\Omega^{+}_{pu}(T(R))$. For $p, u$ fixed, let us denote by $s_{p,u}$ the number of different subsets of the form $\Omega^{+}_{p,u}(T(R))$. So, the number of different subsets of the form $\Omega^{+}(T(R))$ is $\prod_{p,u}s_{p,u}$. It is easy to see that $s_{pu}=d_{pu}(T(L))+1$.
\end{itemize}
\end{proof}

\section{Root of unity case}\label{section-root}
This section is devoted to describing the irreducible module of the quantum enveloping algebra $U_q$ when the complex parameter $q$ is a root of unity. In this case denote by $d$ its order. Since $q \neq \pm 1$. We must have $d>2$.

\begin{theorem}\cite{KS}
When $q$ is a root of unity, any irreducible module of $U_q$ is finite dimensional.
\end{theorem}

Denote

\begin{displaymath}
e = \left\{ \begin{array}{ll}
d & \textrm{if $d$ is odd}\\
d/2 & \textrm{d is even}
\end{array} \right.
\end{displaymath}

It is easy to verify that
$$[x]_q=0\Longleftrightarrow x=0 \text{ mod } e.$$

\begin{remark}
In the Gelfand-Tsetlin formulae (\ref{Gelfand-Tsetlin formulas}), none of the $[l_{ki}-l_{kj}]_q$ is zero if $l_{n1}-l_{nn}\leq e$. So when $q$ is a root of unity, Theorem~\ref{Theorem: quantum GT theorem} holds if
$\lambda_1-\lambda_n\leq e+1-n$.
 For   a generic tableau$T(L)$  all $[l_{ki}-l_{kj}]_q$ are not zero. Hence, Theorem~\ref{Theorem: GT theorem q-generic} holds when $q$ is a root of unity.
\end{remark}

%\begin{proposition}
%Let $T(L)$ be the tableau satisfying the following defining condition:
%\begin{itemize}
%\item[(i)] All the small tableaux in $T(L)$ are standard.
%\item[(ii)]$l_{ij}-l_{ik}\notin \frac{1(q)}{2}+ \mathbb{Z}$ except $l_{ij},l_{ik}$ are in small tableaux.
%\end{itemize}
%By ${\mathcal B}(T(L))$ we will denote the set of all Gelfand-Tsetlin tableaux $T(R)$ of height $n$ satisfying $r_{nj}=l_{nj}$, $r_{ij}-l_{ij}\in\mathbb{Z}$ for $1\leq j\leq i \leq n-1$ such that all the small tableaux are standard. Then $\Span {\mathcal B}(T(L))$ has  a structure of a $U_q$-module with action of the generators of $U_q$ given by the Gelfand-Tsetlin formulae (\ref{Gelfand-Tsetlin formulas}).
%
%\setlength{\unitlength}{1cm}
%\begin{picture}(5,7)
%\thicklines
%\put(0,6){\line(6,0){6}}
%\put(6,6){\line(-3,-5){3}}
%\put(0,6){\line(3,-5){3}}
%\put(2,6){\line(-3,-5){1}}
%\put(4.2,6){\line(-3,-5){1}}
%\put(2.2,6){\line(3,-5){1}}
%\put(4.5,6){\line(3,-5){0.5}}
%\put(5.5,6){\line(-3,-5){0.5}}
%\put(4,5.4){$\cdots$}
%\end{picture}
%\end{proposition}

 Quantum Gelfand-Tsetlin subalgebra $\Gamma_q$  separates the tableaux in the following sense.

\begin{theorem}\label{separation root of unity}
Let $q$ be a root of unity, $T(L)$ a generic tableau. If $T(R), T(S)\in V(T(L))$  and $r_{ij}-s_{ij}\neq 0 \text{ mod }e, 1\leq j \leq i< n$,
then $\Gamma_q$ separates $T(R)$ and $T(S)$.
\end{theorem}
\begin{proof}
Let $T(R)$ and $T(S)$ be two tableaux with two different $m$-th row. Assume $T(R)$ and $T(S)$ have the same eigenvalue for any element in $\Gamma_q$ . It is easy to see from (\ref{eigenvalue}) that $(q^{2s_{m1}},\ldots,q^{2s_{mm}})$ is a permutation of $(q^{2r_{m1}},\ldots,q^{2r_{mm}})$. For any $r_{mi}$, there exist $j$ such that $q^{2r_{mi}}=q^{2s_{mj}}$. We have that $r_{mi}-s_{mj}\in\frac{1(q)}{2}$. $T(L)$ is $q$-generic, one has that $i=j$. Since $r_{ij}-s_{ij}\neq 0 \text{ mod }e$, then $r_{mi}=s_{mj}$ which is a contradiction.
\end{proof}

\begin{proposition}
 Let $T(R)$ be a tableau in $V(T(L))$ and $N$ the submodule of $V(T(L))$ generated by $T(R)$. If $g \cdot T(R) = \sum_i c_i T(R_i)$ for some distinct tableaux $T(R_i)$ in ${\mathcal B}(T(L))$ and nonzero $c_i \in {\mathbb C}$, we have $T(R_i) \in N$ for all $i$.
\end{proposition}
\begin{proof}
Suppose $g=e_{k}$ for some $1\leq k\leq n-1$. By the Gelfand-Tsetlin formulas, we have
$$e_{k}(T(R))=-\sum_{j=1}^{k}\frac{\prod_{i} [r_{k+ 1,i}-r_{k,j}]_q}{\prod_{i\neq j} [r_{k,i}-r_{k,j}]_q}T(R+\delta^{kj})$$

Let $T(R_1)$ and $T(R_2)$ be any two tableaux in the summation with nonzero coefficients, then $(r_1)_{ij}-(r_2)_{ij}=0\text{ or }\pm 1$ for any $1\leq j \leq i< n$. It follows from Theorem \ref{separation root of unity} that $\Gamma_q$ separate these two tableaux. Thus $T(R_i) \in N$ for all $i$.

The proof that  $f_{k}(T(R))\in  W(T(R))$ is analogous to the one of  $e_{k}(T(R))$ . The case
$ q^{\epsilon_{k}}$ is trivial because $ q^{\epsilon_{k}}$ acts as a  multiplication by a scalar on $T(R)$.
\end{proof}

\subsection{Submodule generated by a single tableau}
\begin{notation} \label{notation}
Let $T(R)$ be a fixed tableau of height $n$. We set
\begin{itemize}
\item[(a)] $\omega_{p,s,u}(T(R)):=r_{p,s}-r_{p-1,u}$.
If $\omega_{p,s,u}(T(R))\in \frac{1(q)}{2}+\mathbb{Z}$, we denote $\omega_{p,s,u}(T(R))=u_{p,s,u}(T(R))+v_{p,s,u}(T(R))$, where $u_{p,s,u}(T(R))\in \frac{1(q)}{2}$ and $0\leq v_{p,s,u}(T(R))< e$.
\item[(b)] $\Omega(T(R)):=\{(p,s,u): \omega_{p,s,u}(T(R))\in \frac{1(q)}{2}+\mathbb{Z}\}$
\item[(c)] $\mathcal{N}(T(R)):=\{T(S)\in \mathcal{B}(T(L))\ |\ \omega_{p,s,u}(T(S))-u_{p,s,u}(T(R))\in\mathbb Z_{\geq 0}
$ for all $(p,s,u)\in \Omega(T(R))\}$.
\item[(d)] $ W(T(R)):=\Span \mathcal{N}(T(R))$
\item[(e)] $U_q\cdot T(R)$: the $U_q$-submodule of $V(T(L))$ generated by $T(R)$.
\end{itemize}
\end{notation}

\begin{theorem}\label{submodule}
Let $T(L)$ be a $q$-generic Gelfand-Tsetlin tableau, $T(R)$ and $T(S)$ be in $\mathcal{B}(T(L))$.
\begin{itemize}
\item[(i)]
The Gelfand-Tsetlin formulas endow $ W(T(R))$ with a $U_q$-module structure.
\item[(ii)]
$U_q\cdot T(R)= W(T(R))$. In particular, $\mathcal{N}(T(R))$ forms a basis of   $U_q\cdot T(R)$, and the action of $U_q(\mathfrak{gl}(n))$ on  $U_q\cdot T(R)$ is given by the Gelfand-Tsetlin formulas (\ref{Gelfand-Tsetlin formulas}).
\item[(iii)]
$U_q\cdot T(R)=  U_q\cdot T(S)$ if and only if
 $u_{p,s,u}(T(S))= u_{p,s,u}(T(R))$ for all $(p,s,u)\in \Omega(T(L))$.
\end{itemize}
\end{theorem}
\begin{proof}
(i) In order to prove that $W(T(R))$ is a submodule, it is enough to show $g\cdot T(S)$ is in $W(T(R))$ for every generator of  $U_{q}$. The proof is similar to theorem \ref{basis of submodule generated by T} (i).

(ii)Similar to theorem \ref{basis of submodule generated by T} (ii), it is sufficient to prove that for any $T(S)\in W(T(R))$, $T(S)\preceq_{(p)} T(R)$  for some $p\in \mathbb{Z}_{> 0}$.

(iii)It follows from (i) and (ii).
\end{proof}

\subsection{New constructions of irreducible Gelfand-Tsetlin modules}
In this section we use Gelfand-Tsetlin basis to give a new realistion of some irreducible Gelfand-Tsetlin modules in root of unity case. We assume $d$ to be odd. 

Let $p=(p_{ij}),1\leq j\leq i< n$ with nonzero entries in $\mathbb C$, $W_{ij}(R)$ be the submodule generated by $T(R+d\delta^{ij})-p_{ij}T(R)$. By Theorem \ref{submodule} a basis for $W_{ij}(R)$ is the set $\{T(S+d\delta^{ij})-p_{ij}T(S)\ |\ T(S)\in W(T(R))\}$.

Let $N=\sum\limits_{\substack{T(R)\in B(T(L))\\1\leq j\leq i< n}}W_{ij}(R)$, and $M=V(T(L))/N$.

\begin{theorem}\label{irreducible module 1}
$M$ is an irreducible module with dimension $d^{\frac{n(n-1)}{2}}$. Moreover, $M$ has a basis of tableaux $T(L+m_{ij}\delta^{ij})$, $0\leq m_{ij}< d$, $1\leq j\leq i<n$.
\end{theorem}
\begin{proof}
The submodule $N$ has a basis $\{T(R+\delta^{ij})-p_{ij}T(R):R\in B(T(L)),1\leq j\leq i< n\}$. So the subquotient $M$ has basis $T(L+m_{ij}\delta^{ij})$, $0\leq m_{ij}< d$, $1\leq j\leq i<n$. We denote the basis of $M$ by $I$. Suppose $M_1$ is a nonzero submodule of $M$, by Proposition (\ref{separation root of unity}) the basis of $M_1$ is a subset of $I$. From theorem \ref{submodule} and the relations in quotient module $M$, one has that $I\subseteq U_qT(R)$ for any tableau $T(R)$ in $I$, . Thus $M_1=M$ and $M$ is irreducible.
\end{proof}

\begin{remark}
This module is similar to the module constructed in \S 7.5.5 of \cite{KS}.
\end{remark}
From now on we will denote by $\Lambda$ the following set
$$\{(i,j)\ |\ (i+1,s,j)\in\Omega(T(R)) \text{ for some } 1\leq s \leq i+1\}.$$

\begin{definition}
For any $T(R)\in\mathcal{B}(T(L))$, $1<p\leq n$ and $1\leq u\leq p-1$, for $(i,j)\in \Lambda$ define $a_{ij}(T(R))$ and $b_{ij}(T(R))$ as follows
$$
a_{ij}(T(R))=\min \{v_{i+1,s}\ |\ (i+1,s,j)\in\Omega(T(R))\},
$$
$$
b_{ij}(T(R))=\min \{d-v_{i+1,s}\ |\ (i+1,s,j)\in\Omega(T(R))\}.
$$
Define $$t_{ij}(T(R))=\begin{cases}
a_{ij}(T(R))+b_{ij}(T(R)) &\text{ for } (i,j)\in\Lambda\\
d &\text{ for }(i,j)\notin\Lambda.
\end{cases}
$$
\end{definition}

\begin{definition}
Let $\Lambda_1$ be a subset of $\Lambda$, $\Lambda_2=\Lambda\setminus\Lambda_1$, $\widetilde{M}(T(R))$ de quotient of $U_q\cdot T(R)$ by $$\left(\sum_{(i,j)\notin\Lambda}W_{ij}(R)+\sum_{T(S_1)}U_q(T(S_1))+\sum_{T(S_2)}U_q(T(S_2')-p_{ij}T(S_2))\right),$$
where $T(S_t),t=1,2$ run through over the set of tableaux in $\mathcal{N}(T(R))$ such that $(i,j)\in \Lambda_t$,
$\omega_{i-1,s,j}T(S_2')-\omega_{i-1,s,j}(T(S_2))=d$ for some $(i-1,s,j)\in\Omega(T(R))$, 
$\omega_{p,s,u}T(S_2)-\omega_{p,s,u}(T(S_2'))=0$
for any $(p,s,u)\neq (i-1,s,j)$.
\end{definition}

\begin{theorem}\label{irreducible module 2}
$\widetilde{M}(T(R))$ is an irreducible module of dimension $\prod\limits_{{1\leq j\leq i < n}}t_{ij}(T(R))$.
\end{theorem}
\begin{proof}
The subquotient  $U_qT(R)/\sum_{T(S)}U_q(T(S)))$ has basis
$$I=\{T(S)\ |\ u_{p,s,u}(T(S))= u_{p,s,u}(T(R)) \text{ for all }(p,s,u)\in \Omega(T(L)\}.$$
The module $\widetilde{M}(T(R))$ can be regarded as the subquotient of $U_qT(R)/\sum_{T(S)}U_q(T(S))$.
Then it has basis:
$\{T(S)\in I\ |\ s_{ij}=r_{ij}+m_{ij}, 0\leq m_{ij}<d,(i,j)\notin \Lambda \}$. Similar to theorem \ref{irreducible module 1}  $\widetilde{M}(T(R))$ is irreducible. For any $(i,j)\in \Lambda$, if we fix the $i+1$-th row of the tableau, the number of distinct $s_{ij}$ in $I$ is $t_{ij}(T(R))$. For $(i,j)\notin \Lambda$, there are $d$ different $s_{ij}$. Thus the dimension of $\widetilde{M}(T(R))$ is $\prod\limits_{{1\leq j\leq i < n}}t_{ij}(T(R))$.

\end{proof}

\subsection{Example}

Recall two families $d$-dimensional modules of $U_q(sl_2)$\cite{K}. The first depends on three complex numbers $\lambda, a$ and $b$. We assume $\lambda\neq0$. Consider the $d$-dimensional vector space with a basis
$\{v_0, v_1, \ldots ,v_{d-1} \}$. for $0\leq p\leq d-1$, set
\begin{align}
Kv_{p}=\lambda q^{-2p}v_{p},\\
Ev_{p+1}=(\frac{q^{-p}\lambda-q^{p}\lambda^{-1}}{q-q^{-1}}[p+1]_q+ab)v_{p},\\
F_{v_p}=v_{p+1},
\end{align}
and $Ev_{0}=av_{d-1}, Fv_{d-1}=bv_{0}$, and $Kv_{e-1}=\lambda q^{-2(d-1)}v_{p}$. These formula endow the vector space with a $U_q$-module structure, denoted by $V(\lambda, a, b)$.

The second family depends on two scalars $\mu\neq 0$ and $c$. Let $E, F, K$ act on the vector space with basis
$\{v_0, v_1, \ldots ,v_{d-1} \}$ by
\begin{align}
Kv_{p}=\mu q^{2p}v_{p},\\
Fv_{p+1}=\frac{q^{-p}\mu^{-1}-q^{p}\mu}{q-q^{-1}}[p+1]_qv_{p},\\
E_{v_p}=v_{p+1},
\end{align}
and $Fv_{0}=0, Ev_{d-1}=cv_{0}$, and $Kv_{e-1}=\mu q^{-2}v_{e-1}$. These formula endow the vector space with a $U_q$-module structure, denoted by $\tilde{V}(\mu, c)$.

\begin{theorem}\cite{K}
Any irreducible $U_q$ module of dimension $d$ is isomorphic to one of the following list:
\begin{itemize}
\item[(i)]$V(\lambda, a, b)$ with $b\neq 0$,
\item[(ii)] $V(\lambda, a, 0)$  where $\lambda$ is not of the form $\pm q^{j-1}$ for any $1\leq j\leq d-1$,
\item[(iii)] $\tilde{V}(\pm q^{1-j}, c)$ with $c\neq 0$ and $1\leq j\leq d-1$.
\end{itemize}
\end{theorem}

In the following we will compare  above modules with modules in \ref{irreducible module 1} and \ref{irreducible module 2}.
Let $x, y, z$ be three complex number, $v_p=(x, y|z-p), 0\leq p\leq d-1$. Consider the vector space with basis of tableaux $\{T(v_p):0\leq p\leq d-1\}$. Theorem \ref{irreducible module 1} endows the vector space with a $U_q$-module structure. The actions of $E, F, K$ are given by
\begin{align}
KT(v_p)=q^{2z-(x+y+1)}q^{-2p}T(v_p),\\
ET(v_{p+1})=-[x+p+1-z]_q[y+p+1-z]_qT(v_p),\\
FT(v_p)=T(v_{p+1}),
\end{align}
and $ET(v_{0})=-s[x-z]_q[y-z]_qT(v_{d-1}), FT(v_{d-1})=\frac{1}{s}T(v_{0})$.
Let $\lambda=q^{2z-(x+y+1)}, b=\frac{1}{s}, a=-s[x-z]_q[y-z]_qv_{d-1}$, this module is isomorphic to
$V(\lambda, a, b)$ with $b \neq 0$.

Let $x, y, z$ be three complex number with $x-z\text{ or } y-z\in \frac{1(q)}{2}$. Consider be the vector space with basis of tableaux $\{T(v_p):0\leq p\leq d-1\}$, where $v_p=(x, y|z-p), 0\leq p\leq d-1$. Theorem \ref{irreducible module 2} endows the vector space with a $U_q$-module structure. The actions of $E, F, K$ are given by
\begin{align}
KT(v_p)=q^{2z-(x+y+1)}q^{-2p}T(v_p),\\
ET(v_{p+1})=-[x+p+1-z]_q[y+p+1-z]_qT(v_p),\\
FT(v_p)=T(v_{p+1}),
\end{align}
and $Ev_{0}=0, Fv_{d-1}=sv_{0}$.
This module is isomorphic to
$V(\lambda, 0,s), \lambda=q^{2z-(x+y+1)}$.

There exist an algebra endomorphism of $U_q(sl_2)$ such that $E\longmapsto F, F\longmapsto E, K\longmapsto K^{-1}$.
$V(\lambda, a,0)$ and $\tilde{V}(\mu, c)$ can be obtained from $V(\lambda, 0,b)$ by the algebra endomorphism.

%%%%%%%%%%%%%%%%%%%%%%%%%%%%%%%%%%%%%%%%%%%%%%%%%%%%%%%%%%%%%%%%%%%%%%%%%%%%%%%%%%%%%%%%%%%%%%%%%%%%%%%%%%%%

\end{document}